\documentclass[oneside,12pt]{amsart}
\usepackage{amsmath, amsfonts,amsthm,times,graphics}

 \makeatletter
\renewcommand*\subjclass[2][2010]{%
  \def\@subjclass{#2}%
  \@ifundefined{subjclassname@#1}{%
    \ClassWarning{\@classname}{Unknown edition (#1) of Mathematics
      Subject Classification; using '2010'.}%
  }{%
    \@xp\let\@xp\subjclassname\csname subjclassname@#1\endcsname
  }%
}

 \makeatother
\newtheorem{theorem}{Theorem}
\newtheorem{lemma}{Lemma}

\newtheorem{proposition}{Proposition}
\theoremstyle{definition}

 \makeatletter
\renewcommand*\subjclass[2][2010]{%
  \def\@subjclass{#2}%
  \@ifundefined{subjclassname@#1}{%
    \ClassWarning{\@classname}{Unknown edition (#1) of Mathematics
      Subject Classification; using '1991\'.}%
  }{%
    \@xp\let\@xp\subjclassname\csname subjclassname@#1\endcsname
  }%
}
 \makeatother

\begin{document}

\title[A Primality criterion...]{A primality criterion based on a 
Lucas' congruence}

\author{Romeo Me\v strovi\' c}

\address{Maritime Faculty, University of Montenegro, Dobrota 36,
 85330 Kotor, Montenegro} \email{romeo@ac.me}

{\renewcommand{\thefootnote}{}
\setcounter{footnote}{0}}

\maketitle

 \begin{abstract} Let $p$ be a prime.
In 1878  \'{E}.  Lucas proved that the congruence
$$
{p-1\choose k}\equiv (-1)^k\pmod{p}
 $$
holds for any nonnegative integer $k\in\{0,1,\ldots,p-1\}$. 
The converse statement was given in Problem 1494  of 
{\it Mathematics Magazine} proposed  in 1997 by E. Deutsch and I.M. Gessel. 
In this note we generalize this converse assertion  
by the following result: 
If $n>1$ and $q>1$ are integers such that 
  $$
{n-1\choose k}\equiv (-1)^k \pmod{q}
   $$
for every integer $k\in\{0,1,\ldots, n-1\}$,
then $q$ is a prime and $n$ is a power of $q$.
  \end{abstract} 

\section{Introduction and the main result}

As noticed in  \cite{gr}, many great mathematicians of the nineteenth century
considered problems involving binomial coefficients modulo prime or  
 prime power (for instance Babbage, Cauchy, Cayley, Gauss, Hensel,
Hermite, Kummer, Legendre, Lucas, and Stickelberger). They discovered 
a variety of elegant and surprising theorems which are
often easy to prove. For more information on these classical  results, their
extensions,  and new results about this subject, see Dickson \cite{d},
and Granville \cite{gr}. Furthermore, 
the arithmetic and divisibility properties of binomial coefficients can be 
used to establish criteria for primality.
In 1801 Gauss \cite[Disquisitiones Arithmeticae, 1801, art. 329]{ga} wrote:

{\it ``The problem of distinguishing prime numbers from composite numbers ...
is known to be one of the most important and useful in arithmetic. ...
The dignity of the science itself seems to require that every possible means be
explored for solution of a problem so elegant and so celebrated.''}

As far back as 1819 (see \cite{gr}), 
creator of of machines that were precursors of the modern computer
Charles Babbage gave an easily proved 
characterization of the primes as follows:
an integer $n\ge 2$ is a prime if and only if
${n+m\choose n}\equiv 1(\bmod{\,n})$ for all positive integers $m$ 
with $0\le m\le n-1$. This criterion was recently generalized by the 
author of this note in \cite[Theorem 1.1]{me3}.  
In 1954 the amateur mathematician Pedro A. Piza \cite{pi} proved that 
an odd positive integer $2n+1\ge 3$ is a prime if and only if 
${2n-k\choose k-1}\equiv 0(\bmod{\,k})$
for all $k\in\{1,2,\ldots,n\}$. In 1972 H.B. Mann and D. Shanks
\cite{ms} discovered another attractive primality criterion 
which may be stated as follows: a positive integer $k\ge 2$ is prime if and 
only if ${n\choose k-2n}\equiv 0(\bmod{\,k})$
for each $n\ge 1$ such that $k/3\le n\le k/2$. The following ``dual''
criterion to that of Mann and Shanks was discovered in 1985  by 
H.W. Gould and W.E. Greig \cite{gg} in the following form: 
a positive integer $k\ge 2$ is a prime if and only if 
${-n\choose k-2n}\equiv 0(\bmod{\,k})$ for each $n\ge 1$ such that $n\le k/2$.
Notice that by the famous Lucas' theorem \cite{l} 
given by the congruence \eqref{con4}, we immediately have
${np \choose mp}\equiv {n \choose m}(\bmod{\,p})$,
where $p$ is a prime,  $n$ and $m$ are integers with $0\le m\le n$. 
In 2009 the author of this note \cite[Theorem]{me1} proved a partial converse 
theorem of this assertion as follows:
If $d,q>1$ are integers such that 
  ${nd\choose md}\equiv {n\choose m} (\bmod{\,q})$
for every pair of integers $n\ge m\ge 0$,
then $d$ and $q$ are powers of the same prime $p$.
 
Let $p$ be a prime. In 1878 \'{E}. Lucas \cite{l} proved that 
  \begin{equation}\label{con1}
{p-1\choose k}\equiv (-1)^k\pmod{p}
  \end{equation}
for any nonnegative integer $k\in\{0,1,\ldots,p-1\}$. 
By Problem 1494  of {\it Mathematics Magazine} proposed by E. Deutsch and 
I.M. Gessel in 1997 \cite{dg} (see also \cite[pp. 277--278]{cg}), 
a converse assertion is also true; that is, 
an integer $p\ge 2$ is a prime if and only if the congruence 
\eqref{con1} holds for each   $k\in\{0,1,\ldots,p-1\}$. 
T.-X. Cai and A. Granville \cite[p. 277]{cg} proved that in this assertion the 
range for $k$ may be shortened to $0\le k\le \sqrt{n}$.
Accordingly, the simultaneous congruences \eqref{con1} with 
$k\in\{0,1,\ldots,p-1\}$ could be used to identify primes.

Our Theorem \ref{th} generalizes the previously mentioned 
criterion for primality. This is motivated by the following 
result.
   \begin{proposition}\label{pr} 
Let $p$ be  a prime and let  $f$ be a positive integer. Then for  each 
$k\in\{0,1,\ldots, p^f-1\}$ we have
 \begin{equation}\label{con2}
{p^f-1\choose k}\equiv (-1)^k\pmod{p}.
  \end{equation}
  \end{proposition}

 We are now ready to state the main result.
 \begin{theorem}\label{th}
Let $n>1$ and $q>1$ be integers such that 
  \begin{equation}\label{con3}
{n-1\choose k}\equiv (-1)^k \pmod{q}
   \end{equation}
for every integer $k\in\{0,1,\ldots, n-1\}$.
Then $q$ is a prime and $n$ is a power of this prime $q$.
   \end{theorem}

\section{Proofs of Proposition \ref{pr} and Theorem \ref{th}}

 \begin{proof}[Proof of Proposition \ref{pr}]
If $a=a_0+a_1p+\cdots +a_lp^l$ and
$b=b_0+b_1p+\cdots +b_lp^l$ are the $p$-adic expansions of 
nonnegative integers $a$ and $b$
(so that $0\le a_i,b_i\le p-1$ for all $i=0,1,\ldots,l$),
then by Lucas's theorem (\cite{l}; also see \cite{gr} or \cite{me2}),
 \begin{equation}\label{con4}
{a\choose b}\equiv \prod_{i=0}^{l}{a_i\choose b_i}\pmod{p}.
 \end{equation}  
If we take $k=\sum_{i=0}^{f-1}k_ip^i$ with $0\le k_i\le p-1$ for 
all $i=0,1,\ldots,f-1$, then in view of the fact  
$p^f-1=\sum_{i=0}^{f-1}(p-1)p^i$,  the  congruences \eqref{con4} and 
\eqref{con1} immediately yield
  \begin{equation}\label{con5}\begin{split}
{p^f-1\choose k}&={\sum_{i=0}^{f-1}(p-1)p^i\choose\sum_{i=0}^{f-1}k_ip^i}
\equiv\prod_{i=0}^{f-1}{p-1\choose k_i}\pmod{p}\\
&\equiv\prod_{i=0}^{f-1}(-1)^{k_i}=(-1)^{\sum_{i=0}^{f-1}k_i}\equiv
(-1)^k\pmod{p}.
 \end{split}\end{equation}

Notice that in the last congruence of \eqref{con5} we have used the fact
that if $p$ is an odd prime, then $k$  and the sum $\sum_{i=0}^{f-1}k_i$
have the same parity, while for $p=2$ holds  $1\equiv -1(\bmod{\, 2})$.
  \end{proof}

Proof of Theorem \ref{th} is based on Proposition \ref{pr} and the following 
lemma. 
  \begin{lemma}\label{le}
Let $p$ be a prime and let $f$ be a positive integer greater than
$1$. Then
      \begin{equation}\label{con6}
{p^f-1\choose p^{f-1}}\equiv\left\{
   \begin{array}{ll}
p-1 &\pmod{p^2}\quad {\rm if}\,\, p\ge 3\\
3 &\pmod{4}\quad {\rm if}\,\, p= 2.
    \end{array}\right.
     \end{equation}
  \end{lemma}
   \begin{proof}
By using the identities ${a-1\choose b}=\frac{a-b}{a}{a\choose b}$
and  ${a\choose b}=\frac{a}{b}{a-1\choose b-1}$ with 
$1\le b\le a$, we have
  \begin{equation}\label{con7}\begin{split}
{p^f-1\choose p^{f-1}}&=\frac{p^f-p^{f-1}}{p^f}{p^f\choose p^{f-1}}=
\frac{p^f-p^{f-1}}{p^f}\cdot 
\frac{p^f}{p^{f-1}}{p^f-1\choose p^{f-1}-1}\\
&=(p-1){p^f-1\choose p^{f-1}-1}.
 \end{split}\end{equation}
Further,  we have
      \begin{equation}\label{con8}\begin{split}
{p^f-1\choose p^{f-1}-1}&=\prod_{i=1}^{p^{f-1}-1}\frac{p^f-i}{p^{f-1}-i}=
\prod_{1\le i\le p^{f-1}-1\atop i\not\equiv 0(\bmod{\,p})}\frac{p^f-i}{p^{f-1}-i}
\prod_{1\le i\le p^{f-1}-1\atop i\equiv 0(\bmod{\,p})}\frac{p^f-i}{p^{f-1}-i}. 
        \end{split}\end{equation}
If $f\ge 3$, then  
 \begin{equation}\label{con9}
\frac{p^f-i}{p^{f-1}-i}\equiv  1(\bmod{\,p^2})\quad 
{\rm for\,\, each}\,\, i\,\, {\rm such\,\,that}\,\, 
1\le i\le p^{f-1}-1\,\,{\rm and}\,\, i\not\equiv 0(\bmod{\,p}).
    \end{equation}
Furthermore, for  $f\ge 3$ we have  
  \begin{equation}\label{con10}\begin{split}
\prod_{1\le i\le p^{f-1}-1\atop i\equiv 0(\bmod{\,p})}\frac{p^f-i}{p^{f-1}-i}
&=\prod_{j=1}^{p^{f-2}-1}\frac{p^f-jp}{p^{f-1}-jp}= 
\prod_{j=1}^{p^{f-2}-1}\frac{p^{f-1}-j}{p^{f-2}-j}=
{p^{f-1}-1\choose p^{f-2}-1}.
       \end{split}\end{equation}
Substituting \eqref{con9} and \eqref{con10}  into 
\eqref{con8} we find that for each prime $p$ and every integer $f\ge 3$
holds 
 \begin{equation}\label{con11}
{p^f-1\choose p^{f-1}-1}\equiv {p^{f-1}-1\choose p^{f-2}-1}\pmod{p^2}.
    \end{equation}
Iterating the congruence \eqref{con11} $f-2$ times yields 
   \begin{equation}\label{con12}
{p^f-1\choose p^{f-1}-1}\equiv {p^2-1\choose p-1}\pmod{p^2},
    \end{equation}
which substituting into \eqref{con7} for every $f\ge 3$ gives 
 \begin{equation}\label{con13}
{p^f-1\choose p^{f-1}}\equiv (p-1){p^2-1\choose p-1}\pmod{p^2}.
    \end{equation}
Further, for each prime $p\ge 3$ we have
    \begin{equation}\label{con14}\begin{split}
{p^2-1\choose p-1}&=\prod_{i=1}^{p-1}\frac{p^2-i}{p-i}\equiv 
\prod_{i=1}^{p-1}\frac{-i}{p-i}\pmod{p^2}\\
&\equiv\prod_{i=1}^{p-1}\frac{-i(p+i)}{-i^2}=
\prod_{i=1}^{p-1}\left(\frac{p}{i}+1\right)\equiv 1+p\sum_{i=1}^{p-1}
\frac{1}{i}\pmod{p^2}\\
&=1+p\sum_{i=1}^{(p-1)/2}\left(\frac{1}{i}+\frac{1}{p-i}\right)=
1+p^2\sum_{i=1}^{(p-1)/2}\frac{1}{i(p-i)}\equiv 1\pmod{p^2}.
    \end{split}\end{equation}
Substituting \eqref{con14} into  \eqref{con13} we obtain 
that for each $f\ge 3$ and any prime $p\ge 3$
 \begin{equation}\label{con15}
{p^f-1\choose p^{f-1}}\equiv p-1\pmod{p^2}.
    \end{equation}
Notice also that for a prime  $p\ge 3$ the identity 
\eqref{con7} with $f=2$ and the congruence \eqref{con14} yield
    \begin{equation}\label{con16}
{p^2-1\choose p}=(p-1){p^2-1\choose p-1}\equiv p-1\pmod{p^2}.
    \end{equation}
The congruences  \eqref{con15} and  \eqref{con16}
imply the first part of the congruence \eqref{con6}.

It remains to consider the case when $p=2$. By \eqref{con12}, 
for each $f\ge 3$ we have 
     \begin{equation}\label{con17}
{2^f-1\choose 2^{f-1}}\equiv 3\pmod{4},
    \end{equation}
which is also satisfied for $f=2$. The congruence \eqref{con17}
is in fact the second part of the congruence \eqref{con6}, and the 
proof is completed.
 \end{proof}

\begin{proof}[Proof of Theorem \ref{th}]
Taking $k=1$ into the congruence \eqref{con3} we obtain 
  \begin{equation}\label{con18}
n\equiv 0(\bmod{\,q}).
   \end{equation}
Therefore, if $p$ is a prime divisor of $q$, then 
$n$ can be expressed as $n=sp^f$, where  $f$ and $s$ 
are positive  integers such that $s$ is not divisible by $p$. 
Now we consider the following  three cases.

{\it Case} 1: $s=f=1$. Then $n=p$, and this together 
with the congruence \eqref{con18} yields $q=p$.

{\it Case} 2: $s=1$ and $f\ge 2$. Then $n=p^f$, and 
hence, by the congruence \eqref{con18} it follows that 
$q=p^e$ with  $1\le e\le f$. By the congruence \eqref{con6} of 
Lemma \ref{le} we have
    \begin{equation}\label{con19}
{n-1\choose p^{f-1}}={p^f-1\choose p^{f-1}}\equiv\left\{
   \begin{array}{ll}
p-1 &\pmod{p^2}\quad {\rm if}\,\, p\ge 3\\
3 &\pmod{4}\quad {\rm if}\,\, p= 2.
    \end{array}\right.
     \end{equation}
On the other hand, if we suppose that $e\ge 2$, then  the congruence 
\eqref{con3} with $k=p^{f-1}$ reduced modulo $p^2$ yields  
      \begin{equation}\label{con20}
{n-1\choose p^{f-1}}={p^f-1\choose p^{f-1}}\equiv\left\{
   \begin{array}{ll}
-1 &\pmod{p^2}\quad {\rm if}\,\, p\ge 3\\
1 &\pmod{4}\quad {\rm if}\,\, p= 2.
    \end{array}\right.
     \end{equation}
Comparing the congruences \eqref{con19}  and \eqref{con20}, we get 
$p\equiv 0(\bmod{\, p^2})$. This contradiction 
shows  that must be $e=1$, or equivalently, $q=p$.

 {\it Case} 3: $s\ge 2$. Then take $s=\sum_{i=0}^ts_ip^i$ with 
$0\le s_i\le p-1$ for all $i=1,\ldots,t$ and $1\le s_0\le p-1$.
Applying  Lucas' theorem (the congruence \eqref{con4}), we have 
    \begin{equation}\label{con21}\begin{split}
{n-1\choose p^f}&={sp^f-1\choose p^f}
={\sum_{i=1}^ts_ip^{i+f}+(s_0-1)p^f+\sum_{i=0}^{f-1}(p-1)p^i
 \choose p^f}\\
&\equiv {s_0-1\choose 1}=s_0-1\equiv s-1\pmod{p}.
      \end{split}\end{equation}   
The congruence \eqref{con21} and the condition \eqref{con3} with
$k=p^f$ imply  that  $s-1\equiv (-1)^{p^f}(\bmod{\,p})$. This shows that 
 must be $s\equiv 0(\bmod{\,p})$. A contradiction, and thus 
this case is impossible.

Finally, the all three considered cases clearly  completes the proof of 
the theorem.
 \end{proof}

\end{document}